\documentclass[10pt]{amsart}
\usepackage{amsmath,amsthm}
\usepackage{mathabx}
\usepackage[colorlinks=true,linkcolor=blue,urlcolor=blue,citecolor=blue]{hyperref}
\usepackage[all]{xy}
\usepackage[usenames,dvipsnames,svgnames,table]{xcolor}
\usepackage[latin1]{inputenc}

\newtheorem{theorem}{Theorem}[section]
\newtheorem{lemma}[theorem]{Lemma}
\newtheorem{proposition}[theorem]{Proposition}

\newtheorem{definition}[theorem]{Definition}

\theoremstyle{remark}

\numberwithin{equation}{section}


\title
{{globally diffeomorphic $\sigma$--harmonic mappings}}

\author[G. Alessandrini]
{Giovanni Alessandrini}
\address[G. Alessandrini]{Dipartimento di Matematica e Geoscienze, Università  di Trieste, Via Valerio 12/b, 34100 Trieste, Italia} 
\email[G. Alessandrini]{alessang@units.it}
\author[V. Nesi]
{Vincenzo Nesi}
\address[V.  Nesi]{Dipartimento di Matematica ``G. Castelnuovo", Sapienza, Università di Roma,
Piazzale A. Moro 2, 00185 Roma, Italy}
\email[V. Nesi]{nesi@mat.uniroma1.it}

\begin{document}
\begin{abstract} Given a two--dimensional mapping $U$ whose components solve a divergence structure elliptic equation, we give necessary and sufficient conditions on the boundary so that $U$ is a global diffeomorphism.
\end{abstract}
\maketitle

\section{Introduction}
\label{intro} Let $B = \{(x,y)\in \mathbb R^2: x^2+y^2<1\}$ denote
the unit disk. 
We denote by  $\sigma=\sigma(x)$, $x\in B$, a possibly non--symmetric matrix  having measurable entries and satisfying the ellipticity conditions
\begin{equation}\label{ell}
\begin{array}{ccrllll}
\hbox{$\hskip0,4cm \sigma(x) \xi\cdot \xi \geq  K^{-1} |\xi|^2$, for every $\xi \in \mathbb R^2\ , x \in B$\,,}\\

\hbox{$\sigma^{-1}(x) \xi \cdot \xi \geq  K^{-1} |\xi|^2$, for every $\xi \in \mathbb R^2 \ , x \in B$\,,}

\end{array}
\end{equation}
for a given constant $K\ge 1$.

Given a diffeomorphism $\Phi=(\varphi^1,\varphi^2)$ from the unit circle
$\partial B$ onto a simple closed curve $\gamma\subseteq \mathbb
R^2$, we denote by $D$  the bounded domain such that $\partial D = \gamma$. With no loss of generality, we may assume that $\Phi$ is orientation preserving.

Let us consider the  mapping $U=(u^1,u^2)\in W^{1,2} (B;\mathbb R^2)\cap
C(\overline{B};\mathbb R^2)$ whose components are the solutions to the following Dirichlet problems
\begin{equation}\label{main}
\left\{
\begin{array}{lll}
{\rm div} (\sigma \nabla u^i)=0,&\hbox{in}&B,\\
u^i=\varphi^i,&\hbox{on}&\partial B \ , i=1,2 \ .
\end{array}
\right.
\end{equation}
Loosely speaking, the question that we intend to address here is:

\emph{Under which conditions can we assure that $U$ is an invertible mapping between   $B$ and $D$ (or $\overline{B}$ and  $\overline{D}$)?}

The classical starting point for this issue is the celebrated Radò--Kneser--Choquet Theorem \cite{rado, k, Choquet, duren} which asserts that assuming $\sigma= I$, the identity matrix, (that is: $u^1, u^2$ are harmonic) if $D$ is convex  then $U$ is a homeomorphism. Generalizations to equations with variable coefficients have been obtained in \cite{bmn, ANARMA}  and to certain nonlinear systems in \cite{BP, AS, IKO}. Counterxamples \cite{Choquet, ANPISA} show that if $D$ is not convex then the invertibility of $U$ may fail.

In \cite{ANPISA} the present authors investigated, in the case of harmonic mappings, which additional conditions are needed for invertibility in the case of a possibly non--convex target $D$. In particular, in \cite[Theorem 1.3]{ANPISA} it is proven that, assuming $\sigma= I$, $U$ is a diffeomorphism if and only if $ \det DU >0$ everywhere on $\partial B$. An improvement to this result, still in the harmonic case, is due to Kalaj \cite{kalaj}.

Here we intend to treat the case of equations with variable coefficients. The main result in this note is the following:

\begin{theorem}\label{main.th}
Assume that the entries of $\sigma$ satisfy $\sigma_{ij}\in C^{\alpha}(\overline{B})$ for some $\alpha \in (0,1)$ and for every $i,j=1,2$.
Assume, in addition, that
 $U\in C^1(\overline{B};\mathbb R^2)$.

The mapping $U$ is a diffeomorphism of $\overline{B}$ onto
$\overline{D}$ \emph{if and only} if
 \begin{equation}\label{d>0}
 \begin{array}{lll}
 \det DU >0 &\hbox{everywhere on}&\partial B.
 \end{array}
 \end{equation}
\end{theorem}
It is evident that, if $U$ is a diffeomorphism on $\overline{B}$, then  $\det DU \neq 0$ on $\partial B$. Thus, from now on, we shall focus on the reverse implication only.

In Section \ref{S2} we begin by proving Theorem \ref{smooth.th}, that is, a version of Theorem \ref{main.th} which requires stronger regularity on $\sigma$ and on $\Phi$.

 Section \ref{S3}  contains the completion of the proof of Theorem \ref{main.th}, let us mention that, as an intermediate step, we also prove Theorem \ref{hom}, which treats the case when the Dirichlet data $\Phi$ is merely a homeomorphism, extending to the case of variable coefficients the result proved in \cite[Theorem 1.7]{ANPISA} for the case of $\sigma= I$.
 
 In the final Section \ref{S4} we sketch the arguments for an improvement, Theorem \ref{hopf}, to Theorem \ref{main.th}, in analogy with \cite[Theorem 5.2]{ANPISA}.

\section{A smoother case}\label{S2}
\begin{theorem}\label{smooth.th}
In addition to the hypotheses of Theorem \ref{main.th}, let us assume that the entries of $\sigma$ satisfy $\sigma_{ij}\in C^{0,1}(\overline{B})$and that $\Phi=(\varphi^1,\varphi^2)\in C^{1,\alpha}(\partial B, \mathbb R^2)$ for some $\alpha \in (0,1)$.
 If
\begin{equation}\label{d>0.2}
 \begin{array}{lll}
 \det DU >0 &\hbox{everywhere on}&\partial B \ ,
 \end{array}
 \end{equation}
then the mapping $U$ is a diffeomorphism of $\overline{B}$ onto
$\overline{D}$.
\end{theorem}
We observe that,  assuming  that  $\sigma_{ij}$ are Lipschitz continuous in $\overline{B}$, it is a straightforward matter to rewrite the equation
\[
{\rm div} (\sigma \nabla u)=0 \]
in the form
\begin{equation}\label{Ab}
{\rm div} (A \nabla u) + b\cdot \nabla u=0 \ ,
\end{equation}
where $b=(b^1, b^2)$ is in $L^{\infty}$ and $A$ is a uniformly elliptic symmetric matrix in the sense of \eqref{ell}, with Lipschitz entries, and it satisfies ${\rm det} A =1$ everywhere.

The calculation is as follows. Denote
\[\widehat{\sigma}= \frac{1}{2}(\sigma+\sigma^T) \ , \widecheck{\sigma}= \frac{1}{2}(\sigma-\sigma^T) \ , \] 
where $(\cdot)^T$ denotes the transposition.
Writing the equation in weak form and using smooth test functions, we obtain
\[0=
{\rm div} (\sigma \nabla u)={\rm div} (\widehat{\sigma} \nabla u)+ \partial_{x_i}\widecheck{\sigma}_{ij}\partial_{x_j}u \ ,\]
next we pose $\gamma = \sqrt{{\rm det} \widehat{\sigma}}$ and $A= \frac{1}{\gamma} \widehat{\sigma}$ and we compute
\[0
=\gamma {\rm div} (A \nabla u)+ \partial_{x_i}(\gamma \delta_{ij}+\widecheck{\sigma}_{ij})\partial_{x_j}u \ ,\]
hence $b^j= \frac{1}{\gamma}\partial_{x_i}(\gamma \delta_{ij}+\widecheck{\sigma}_{ij})$.

We recall that local weak solutions $u$ to \eqref{Ab} are indeed $C^{1,\alpha}$, their critical points are isolated and have finite integral multiplicity. This theory has been developed by R. Magnanini and the first named author \cite{AMPisa}. As a consequence of such a theory, we can state the following auxiliary result. Let us start with some notation.

We denote
\begin{equation}\label{ualpha}
u_{\alpha} = \cos \alpha \, u^1 + \sin \alpha \, u^2 \ , \alpha \in \mathbb R\ ,
\end{equation}
where $u^1, u^2$ are the components of the mapping $U$ appearing in Theorem \ref{main.th}. Next we define
\begin{equation}\label{Malpha}
M_{\alpha} = \text{number of critical points of $u_{\alpha}$ in $B$, counted with their multiplicities}\ .
\end{equation}
Note that, in view of \eqref{d>0}, $M_{\alpha}$ is finite for all $\alpha$.

\begin{proposition}\label{Mconst}
Under the assumptions of Theorem \ref{smooth.th}, we have
\begin{equation}\label{MalphaInt}
M_{\alpha} = \frac{1}{2\pi}\int_{\partial B} {\rm d\,arg}(\partial_z u_{\alpha}) \ ,
\end{equation}
moreover $M_{\alpha}=M$ is constant with respect to $\alpha$.
\end{proposition}
Here $\partial_z$ denotes the usual complex derivative, where it is understood $z=x_1+ix_2$.
\begin{proof}
Formula \eqref{MalphaInt} is a manifestation of the argument principle. A proof, with some changes in notation, can be found in \cite[Proof of Theorem 2.1]{AMPisa}. Also, a special case of Theorem 2.1 in \cite{AMPisa}, tell us that if $\xi$ is a $C^1$ unitary  vector field on $\partial B$ such that $\nabla u_{\alpha} \cdot \xi > 0 $ everywhere on $\partial B$ then we have
\begin{equation}\label{MalphaIntxi}
M_{\alpha} = \frac{1}{2\pi}\int_{\partial B} {\rm d\,arg}(\xi) \ .
\end{equation}
Let us denote
\begin{equation}\label{xi}
\xi =\frac{1}{|\nabla u_1|}J \nabla u_1 \ .
\end{equation}
where the matrix $J$ represents the counterclockwise $90^{\circ}$ rotation
\begin{equation}\label{J}
 J=\left(
\begin{array}{ccc}
0&-1\\
1&0
\end{array}
\right)\,,
\end{equation}
and we compute
\[\nabla u_{\alpha} \cdot \xi =  \frac{\sin \alpha}{|\nabla u_1|}\nabla u_2 \cdot J \nabla u_1 = \frac{\sin \alpha}{|\nabla u_1|} \det DU
\]
which is positive for all $\alpha \in (0, \pi)$. Hence $M_{\alpha}$ is constant for all $\alpha \in (0, \pi)$, by continuity the same is true 
for all $\alpha \in [0, \pi]$. The proof is complete, by noticing that $u_{\alpha+\pi}= - u_{\alpha}$. 
\end{proof}
Our next goal being to prove that $M=M_{\alpha}=0$, we return to the equation in pure divergence form. Denoting $u=u_{\alpha}$ for any fixed $\alpha$, we have that the equation
\[
{\rm div} (\sigma \nabla u)=0 \]
holds in $B$. It is well--known that
there exists $v \in W^{1,2}(B)$, called the \emph{stream function} of $u$  such that
\begin{equation}\label{CR}
\nabla v =  J \sigma \nabla u \ ,
\end{equation}
where, again, the matrix $J$ denotes the counterclockwise $90^{\circ}$ rotation \eqref{J}, see, for instance, \cite{AMsiam}. Denoting
\begin{equation}\label{QR}
f = u + i v \ ,
\end{equation}
it is well--known that $f$ solves the Beltrami type equation
\begin{equation}\label{B}
\begin{array}{ll}
f_{\bar{z}}=\mu f_z +\nu \overline{f_z}\ & \hbox{in $B$}\ ,
\end{array}
\end{equation}
where, the so called complex dilatations $\mu , \nu$ are given by
\begin{equation}\label{SNU}
\begin{array}{llll}
\mu=\frac{\sigma_{22}-\sigma_{11}-i(\sigma_{12}+\sigma_{21})}{1+{\rm
Tr\,}\sigma +\det \sigma}& \ ,&\nu =\frac{1-\det \sigma
+i(\sigma_{12}-\sigma_{21})}{1+{\rm Tr\,}\sigma +\det \sigma}\ ,
\end{array}
\end{equation}
and satisfy the following
ellipticity condition
\begin{equation}\label{ellQC}
|\mu|+|\nu|\leq k< 1 \,,
\end{equation}
where the constant $k$ only depends on $K$, see \cite[Proposition 1.8]{ANFINNICO} and the notation  ${\rm Tr\,} A$ is used for the trace of a square matrix $A$.

Furthermore, it is also well--known, Bers and Nirenberg \cite{bersni}, Bojarski \cite{BO}, that a $W^{1,2}$ solution to \eqref{B} fulfills the so--called Stoilow representation
\begin{equation}\label{St}
f= F\circ \chi \ ,
\end{equation}
where $F$ is holomorphic and $\chi$ is a quasiconformal homeomorphism, which can be chosen to map $B$ into itself. Moreover, $\chi$ solves the Beltrami equation
\begin{equation}\label{Bchi}
\begin{array}{ll}
\chi_{\bar{z}}=\widetilde{\mu} \chi_z \ & \hbox{in $B$}\ ,
\end{array}
\end{equation}
where $\widetilde{\mu}$ is defined almost everywhere by
\[\widetilde{\mu} = \mu + \frac{\overline{f_z}}{f_z}\nu \  ,  \]
Note that, under the present assumptions,  $\mu , \nu$ are Lipschitz continuous in $\overline B$ and $f$ is in $C^{1,\alpha}(\overline B, \mathbb C)$. 

From now on, for simplicity, we denote by $B_{\rho}$ be the disk of radius $\rho>0$ concentric to $B$.

In  view of   \eqref{d>0}, there exists $0<\rho<1$ such that $\partial_z f \neq 0$ on $\overline B \setminus B_{\rho} $. As a consequence, $\widetilde{\mu}$ is $C^{\alpha}$ on $\overline B \setminus B_{\rho} $, 
and the following Lemma holds.
\begin{lemma}
Under the assumptions of Theorem \ref{smooth.th},  there exists $0<\rho<1$ such that the mapping $\chi$, appearing in \eqref{St}, belongs to $C^{1,\alpha}$, for some $0<\alpha<1$, when restricted to $\overline B \setminus B_{\rho}$.
\end{lemma}
\begin{proof}
For $\rho$ sufficiently close to $1$ we may represent $\chi = \exp \left(\omega\right)$ in the annulus $\overline B \setminus B_{\rho}$. Also, for every determination of $\omega$, we have
\begin{equation}\label{Bomega}
\omega_{\bar{z}}=\widetilde{\mu} \omega_z \ .
\end{equation}
Now, posing $w = \Re e (\omega) = \log |\chi |$, it is well--known that we have 
\[
{\rm div} (\widetilde{\sigma} \nabla w)=0, \text{ in } B \setminus \overline{B_{\rho} } \]
where $\widetilde{\sigma}$ is given by
\begin{equation}\label{sigtill}
 \widetilde{\sigma}=\left(
\begin{array}{ccc}
\displaystyle{\frac{|1-\widetilde{\mu}|^2}{1-|\widetilde{\mu}|^2}}&\displaystyle{-\frac{2\Im m(\widetilde{\mu})}{1-|\widetilde{\mu}|^2}}\\ \\
\displaystyle{-\frac{2\Im m(\widetilde{\mu})}{1-|\widetilde{\mu}|^2}}&\displaystyle{\frac{|1+\widetilde{\mu}|^2}{1-|\widetilde{\mu}|^2}}
\end{array}
\right)\,,
\end{equation}
%
%
%
and satisfies uniform ellipticity conditions of the form \eqref{ell}, see, for instance, \cite{ANFINNICO}. Moreover, $\widetilde{\sigma}$ has H\"older continuous entries in $\overline B \setminus B_{\rho}$. Now, since, trivially, $w=0$ on $\partial B$, then, by standard regularity at the boundary, $w$ is $C^{1,\alpha}$ near $\partial B$. Such a regularity extends to $\omega$, and then to $\chi$, because \eqref{Bomega} can be rewritten as $\nabla \Im m (\omega) = J \widetilde{\sigma} \nabla w$.
\end{proof}
Next we recall the following classical notion, see for instance \cite{wh}.
\begin{definition}\label{wn.def} Given a
closed curve $\gamma$, parametrized by $\Phi \in C^1(\left[0,2\pi\right];\mathbb R^2)$ and such that
\begin{equation*}
\frac{d \Phi}{d \vartheta} \neq 0,\  \textrm{for every} \  \vartheta\in  [0,2 \pi],
\end{equation*}
we define the {\em winding number} of $\gamma$ as the following integer
\begin{equation*}
{\rm WN}(\gamma) =\frac{1}{2\pi} \int_0^{2\pi} {\rm d\,arg} \left(\frac{d \Phi}{d \vartheta}\right).
\end{equation*}
\end{definition}

\begin{proposition}\label{M+1}
Under the previously stated assumptions
\begin{equation*}
{\rm WN}(f(\partial B))= M+1 \ ,
\end{equation*}
with  $M$ as in Proposition \ref{Mconst}.
\end{proposition}
\begin{proof}
With no loss of generality, we may assume $\chi(1)=1$.

We have that, for every $\vartheta \in \mathbb R$,
\[ f(e^{i \vartheta}) = F(e^{i \varphi(\vartheta)}) \]
where
\[e^{i \varphi(\vartheta)} = \chi(e^{i \vartheta}) \]
hence $\varphi$ is a strictly increasing function from $[0,2\pi]$ into itself, with $C^{1,\alpha}$ regularity.
Consequently
\[\frac{1}{2\pi} \int_0^{2\pi} {\rm d\,arg} \left(\frac{d f(e^{i \vartheta})}{d \vartheta}\right) = 
\frac{1}{2\pi} \int_0^{2\pi} {\rm d\,arg} \left(F'(e^{i \varphi(\vartheta)})\right) + \frac{1}{2\pi} \int_0^{2\pi} {\rm d\,arg} \left(e^{i \varphi(\vartheta)}\varphi'(\vartheta)\right) \ .
\]
For the second integral we trivially have
\[ \frac{1}{2\pi} \int_0^{2\pi} {\rm d\,arg} \left(e^{i \varphi(\vartheta)}\varphi'(\vartheta)\right) = 1 \ ,
\]
whereas, by the argument principle, the integral
\[\frac{1}{2\pi} \int_0^{2\pi} {\rm d\,arg} \left(F'(e^{i \varphi(\vartheta)})\right) = \frac{1}{2\pi} \int_{\partial B} {\rm d\,arg} \left(F'(z)\right)\]
equals the number of zeroes of $F'$ when counted with their multiplicities, which coincides with the number of critical points of $u$, again counted with their multiplicities, that is, $M$. This is a consequence of the notions
 of \emph{geometrical critical points} and \emph{geometric index} introduced in \cite[Definition 2.4]{AMsiam}, which in the present circumstances, coincide with the usual concepts of critical points and multiplicity, respectively.\end{proof}
Next we compute:
\begin{proposition}\label{WN=1}
\begin{equation*}
{\rm WN}(f(\partial B))= {\rm WN}(\Phi(\partial B))=1 \ .
\end{equation*}
\end{proposition}
\begin{proof}
We may fix $\alpha=0$, that is, $u=u^1$, and let $v^1$ be its stream function. For every $t\in[0,1]$ let us consider $U_t = (u^1, (1-t)v^1+ t u^2)$. Trivially
\[U_0\approx u^1+i v^1 =  f \ , U_1 = U \ . \]
We compute 
\begin{equation*}
\det DU_t=(1-t) \sigma \nabla u\cdot \nabla u + t \det DU > 0 \ , \text{ on } \partial B \ , \text{ for every } t\in [0,1] \ ,
\end{equation*}
consequently
\begin{equation*}
\beta_t(\vartheta)=\frac{d}{d\vartheta}U_t(e^{i\vartheta}) \ , \text{ for every } t\in [0,1]\ , \ \vartheta \in \left[0,2\pi\right]\ .
\end{equation*}
never vanishes. By homotopic invariance of the winding number, \cite[Theorem 1]{wh}, the thesis follows.
\end{proof}
\begin{proof}[Proof of Theorem \ref{smooth.th}]
Combining Propositions \ref{Mconst}, \ref{M+1} and \ref{WN=1} we deduce that, for all $\alpha$, $\nabla u_{\alpha}$ nowhere vanishes. Hence $\det DU >0$ everywhere in $\overline B$. Hence it is a local diffeomorphism which is one--to--one on the boundary, by the Monodromy Theorem, see for instance \cite[p.175]{kerek}, the thesis follows.
\end{proof}

\section{Proof of Theorem \ref{main.th}}\label{S3}
We start by removing the hypothesis of Lipschitz continuity on $\sigma$, and obtain an intermediate weaker result.
\begin{lemma}\label{Phismooth}
In addition to the hypotheses of Theorem \ref{main.th}, let us assume
 $\Phi=(\varphi^1,\varphi^2)\in C^{1,\alpha}(\partial B, \mathbb R^2)$, for some $\alpha \in (0,1)$.
Then $U$ is locally a homeomorphism in ${B}$.
\end{lemma} 
\begin{proof}
Let $\sigma_{\varepsilon}$ be a family of $C^{\infty}$ mollifications of $\sigma$, which satisfy ellipticity and H\"older regularity uniformly with respect to $\varepsilon$. Let $U_{\varepsilon}$ be the solution to
\begin{equation}\label{moll}
\left\{
\begin{array}{lll}
{\rm div} (\sigma_{\varepsilon} \nabla u_{\varepsilon}^i)=0,&\hbox{in}&B,\\
u_{\varepsilon}^i=\varphi^i,&\hbox{on}&\partial B \ , i=1,2 \ .
\end{array}
\right.
\end{equation}
By regularity theory,  $U_{\varepsilon} \in C^{1,\alpha}(\overline B, \mathbb R^2)$ uniformly with respect to $\varepsilon$, hence, by the Ascoli--Arzelà Theorem, $U_{\varepsilon_n} \to U$ in $C^{1}(\overline B, \mathbb R^2)$ for some sequence $\varepsilon_n \to 0$. Therefore, for $n$ large enough
\[
\det DU_{\varepsilon_n} >0 \text{ everywhere on } \partial B
\]
thus, by Theorem \ref{smooth.th}, $U_{\varepsilon_n}$ is a diffeomorphism of $\overline{B}$ onto
$\overline{D}$. In particular the number $(M_{\varepsilon_n})_{\alpha}$, associated to $U_{\varepsilon_n}$ according to the definition \eqref{Malpha},
equals zero for all $\alpha$ and for $n$ large enough. In view of the stability of the geometric index, established in \cite[Proposition 2.6]{AMsiam}, we have that $u_{\alpha} = \cos \alpha \, u^1 + \sin \alpha \, u^2$ has no (geometrical) critical point in $B$ for any $\alpha$. We may invoke now \cite[Theorem 3]{ANARMA} to obtain that $U$ is locally a homeomorphism in $B$.
\end{proof}
We now recall a variant to the celebrated H. Lewy's Theorem \cite{Lewy}, recently obtained in \cite[Theorem 1.1]{ANroma}. Here $\Omega\subset \mathbb R^2$ is any open set.
\begin{theorem}\label{HL}
Assume that the entries of $\sigma$ satisfy $\sigma_{ij}\in C^{\alpha}_{loc}(\Omega)$ for some $\alpha \in (0,1)$ and for every $i,j=1,2$.
Let $U=(u^1,u^2) \in W^{1,2}_{loc}(\Omega, \mathbb R^2)$  be such that
\begin{equation}\label{basiceq}
{\rm div} (\sigma \nabla u^i)=0 \ , i=1,2,
\end{equation}
weakly in $\Omega$. If $U$ is  locally a homeomorphism, then it is, locally, a  diffeomorphism, that is
\begin{equation}\label{detloc}
{\rm det} DU\neq 0 \hbox{ for every } x \in \Omega\ .
\end{equation}
\end{theorem}
The following Theorem mimics an analogous result obtained for harmonic mappings in \cite{ANPISA}. 
We recall the following definition.
\begin{definition}\label{local}
Given $P\in \overline{B}$, a mapping  $U\in C(\overline{B};\mathbb
R^2)$ is a \emph{local homeomorphism} at $P$ if there exists a
neighborhood $G$ of $P$ such that $U$ is one--to--one on $G\cap
\overline{B}$.
\end{definition}
\begin{theorem}\label{hom}
Let $\Phi:\partial B\to \gamma \subset \mathbb R^2$ be a homeomorphism onto a simple closed curve $\gamma$.
Let $D$ be the bounded domain such that $\partial D=\gamma$.
Let $U\in W^{1,2}_{\rm loc} (B;\mathbb R^2)\cap C(\overline{B};\mathbb R^2)$ be the solution to (\ref{main}). Assume that the entries of $\sigma$ satisfy $\sigma_{ij}\in C^{\alpha}_{loc}({B})$ for some $\alpha \in (0,1)$ and for every $i,j=1,2$.
If, for every $P\in \partial B$,
the mapping $U$ is a local homeomorphism at $P$, then it
 is a homeomorphism of $\overline{B}$ onto $\overline{D}$ and it is a diffeomorphism of $B$ onto $D$ .
\end{theorem}
We first need the following Lemma, which is adapted from  \cite[Lemma 4.1]{ANPISA}. Let us recall that $B_{\rho}$ denotes the disk of radius $\rho>0$ concentric to $B$.
\begin{lemma}\label{top}
Assume $\Phi:\partial B\to \gamma\subset \mathbb R^2$ is a
homeomorphism onto a simple closed curve $\gamma$. Let $U\in W^{1,2}_{loc}
(B;\mathbb R^2)\cap C(\overline{B};\mathbb R^2)$ be  the
solution to   \eqref{main}.  Assume that the entries of $\sigma$ satisfy $\sigma_{ij}\in C^{\alpha}_{loc}({B})$ for some $\alpha \in (0,1)$ and for every $i,j=1,2$. If, in addition, for every $P\in
\partial B$ the mapping $U$ is a local homeomorphism near $P$,
then there exists $\rho\in (0,1)$ such that $U$ is a
diffeomorphism of $B\setminus \overline{B_{\rho}}$ onto
$U\Big(B\setminus \overline{B_{\rho}}\Big)$.
\end{lemma}
\begin{proof}
For every $P\in \partial B$ let
\[s(P) = \sup \left\{s>0 | U \text{ is a  homeomorphism in }  B_{s}(P)\cap\overline{B} \right\} \ , \]
the function $s(P)$ is positive valued and lower semicontinuous hence, by the compactness of $\partial B$, there exists $\delta>0$ such that $s(P)> 2\delta$ for all $P\in \partial B$. Again by compactness,
there exist finitely many points $P_1,\hdots,P_K\in \partial B$ such that
\begin{equation*}
\partial B \subset  \bigcup\limits_{k=1}^K B_{\delta}(P_k),
\end{equation*}
and $U$ is one--to--one on $B_{2 \delta}(P_k)\cap\overline{B}$
for every $k$. Note that there exists $\rho_0\in (0,1)$ such that
\begin{equation*}
\overline{B}\setminus B_{\rho_0}\subset \bigcup\limits_{k=1}^K
B_{\delta}(P_k).
\end{equation*}
Let $P,Q$ be two distinct points in $\overline{B}\setminus
B_{\rho_0}$. If $|P-Q|<\delta$, then there exists $k=1,\hdots,K$
such that $P,Q\in B_{2 \delta}(P_k)$ and, hence, $U(P)\neq U(Q)$.
Assume now $|P-Q|\geq \delta$. Let
\begin{equation*}
P^{\prime}=\frac{P}{|P|}\quad,\quad Q^{\prime}=\frac{Q}{|Q|}.
\end{equation*}
We have $|P-P^{\prime}|<1-\rho,\, |Q-Q^{\prime}|<1-\rho,$ and thus
\begin{equation*}
|P^{\prime}- Q^{\prime}| >|P-Q| - 2(1-\rho) \geq \delta -2(1-\rho).
\end{equation*}
Choosing $\rho_1$, $\rho_0\leq \rho_1<1$ such that $(1-\rho_1)<\frac{\delta}{4}$, we have
$|P^{\prime}- Q^{\prime}|>\frac{\delta}{2}.$ Now we use the fact
that $P^{\prime}$ and $Q^{\prime}$ belong to $\partial B$ and
$\Phi$ is one--to--one to deduce that there exists $c>0$ such that
\begin{equation*}
|\Phi(P^{\prime})- \Phi(Q^{\prime})|   \geq c.
\end{equation*}
Recall that $U$ is uniformly continuous on $\overline{B}$. Denoting by $\omega$ its modulus of continuity, we have
\begin{equation*}
|U(P) - U(Q) |\geq |U(P^{\prime})- U(Q^{\prime})| -2 \omega(1-\rho)=
\end{equation*}
\begin{equation*}
|\Phi(P^{\prime})- \Phi(Q^{\prime})| -2 \omega(1-\rho)\geq c -2 \omega(1-\rho).
\end{equation*}
Choosing $\rho$, $\rho_1\leq \rho<1$, such that $1-\rho<\omega^{-1}\big(\frac c 4\big)$ we obtain
\begin{equation*}
|U(P) - U(Q) |\geq \frac c 2 >0,
\end{equation*}
which implies the injectivity of $U$ in $\overline{B}\setminus
B_{\rho}$. Consequently, by Theorem \ref{HL}, $\det DU\neq 0$
in $B\setminus \overline{B_{\rho}}$ and the thesis follows.
\end{proof}
\begin{proof}[Proof of Theorem \ref{hom}]
 In view of the already quoted Monodromy Theorem, it suffices to show that $\det DU\neq 0$ everywhere in $B$.

For every $r\in (0,1)$, let us write $\Phi^r:\partial B_r\to \mathbb R^2$ to denote the application given by
\begin{equation*}
\Phi^r= U|_{\partial B_r}.
\end{equation*}
It is obvious, by interior regularity of $U$, that $\Phi^r$ belongs to $C^{1,\alpha}$. On the other hand,
by Lemma \ref{top}, there exists $\rho\in (0,1)$ such that for
every $r\in (\rho,1)$ the mapping $\Phi^r:\partial B_r\to
\gamma_r\subset \mathbb R^2$ is a diffeomorphism of $\partial B_r$
onto a simple closed curve $\gamma_r$. Now, when restricted to $\overline{B_r}$, $U$  
solves (\ref{main}) with $\Phi$ replaced by $\Phi^r$, and $B$ by $B_r$ . Then, up to a rescaling of coordinates,
Lemma \ref{Phismooth} is applicable  and we obtain, in combination with Theorem \ref{HL},
\begin{equation*}
\det DU\neq 0, \quad\hbox{ everywhere in } {B_r} \  .
\end{equation*}
Finally, by Lemma \ref{top} we have $\det DU\neq 0$ in $B\setminus
\overline{B_{\rho}(0)}$ so that $\det DU\neq 0$ everywhere in $B$.
\end{proof}
We now conclude the proof of the main Theorem \ref{main.th}.
\begin{proof}[Proof of Theorem \ref{main.th}]
Having assumed $\det DU >0$  on $\partial B$, by continuity, one can find $0<\rho<1$, sufficiently close to $1$ such that $\det DU >0$ on $\overline B \setminus B_{\rho}$. By  Theorem \ref{hom}, we have that $U$ is  a global homeomorphism, and that $\det DU >0$ in $B$. Consequently, $\det DU >0$ on all of $\overline B$ and the thesis follows.
\end{proof}
\section{An improvement}\label{S4}
In accordance with \cite{ANPISA}, we prove a variation of Theorem \ref{main.th}. First, we recall the following:
\begin{definition}\label{nc.def}
Given a Jordan domain $D$, let us denote by ${\rm co} (D)$ its
convex hull. We define the {\em convex part} of $\partial D$ as
the closed set $\gamma_c=\partial D\cap \partial ({\rm co}(D))$.
Consequently we define  the {\em non--convex part} of $\partial D$
as the open subset $\gamma_{nc}=\partial D\setminus \partial ({\rm
co}(D))$.
\end{definition}
\begin{theorem}\label{hopf}
Under the assumptions of Theorem \ref{main.th}, if
\begin{equation}\label{ncd>0}
\det DU>0\quad \hbox{everywhere on $\Phi^{-1}(\gamma_{nc})$},
\end{equation}
where $\gamma_{nc}$ is the set introduced in Definition
\ref{nc.def} above, then  the
mapping $U$ is a diffeomorphism of $\overline{B}$ onto
$\overline{D}$.
\end{theorem}
It is worth noticing that, if $D$ is convex, then the condition \eqref{ncd>0} is void, which agrees with the known adaptations  \cite{bmn, ANARMA} of the well--known Rad\'o--Kneser--Choquet \cite{k} to the equation \eqref{main}.
\begin{proof}
The proof follows the same line of \cite[Theorem 5.2]{ANPISA}, the only change is that the classical Zaremba--Hopf Lemma for harmonic functions must be replaced by its appropriate adaptation to divergence structure equations with H\"older coefficients, which is due to Finn and Gilbarg \cite{F-G}. We omit the details.
\end{proof}
{\bf Acknowledgements} G.A. was supported by 
Universit\`a degli Studi di Trieste FRA 2016, 
V.N. was supported by Fondi di Ateneo Sapienza ``Metodi di Analisi Reale e Armonica per problemi stazionari ed evolutivi''.

\end{document}